\pdfoutput=1 
\documentclass[10pt,a4paper]{article}

\usepackage[T1]{fontenc}
\usepackage[utf8]{inputenc}
\usepackage{amsmath,amscd}
\usepackage{amsfonts}
\usepackage{amssymb}
\usepackage{amsthm}
\usepackage{xcolor}
\usepackage{enumitem}
\usepackage{tikz-cd}
\usepackage{pdfpages}
\usepackage{hyperref}

\newtheorem{theorem}{Theorem}

\newtheorem{lemma}[theorem]{Lemma}
\newtheorem{proposition}[theorem]{Proposition}

\title{Quintic surfaces with 18 cusps}
\author{Lev Borisov and Carlos Rito}
\date{}

\begin{document}
\maketitle

\begin{abstract}
We construct quintic surfaces in the three-dimensional projective space
$\mathbb P^3$ with $18$ ordinary cusps.  Our starting point is the
Barth--Rams description of quintics containing a $3$-divisible set of $12$
cusps.  A specialization in which the two contact cubics are singular along
two skew lines produces a family with $16$ cusps, and examples with $18$
cusps can be found over small finite fields.  Our main construction
is based on quintics admitting two Barth--Rams decompositions.
The corresponding sets of $12$ cusps meet in $7$ points,
and we prove that the locus of quintics admitting two such decompositions
contains a $6$-dimensional component in the moduli space whose general member
has $17$ cusps. This makes it possible to find members with $18$ cusps
efficiently over finite fields.
We lift one of these surfaces to characteristic
zero using Newton--Hensel lifting and LLL reconstruction, obtaining a
quintic over a number field of degree $22$. We verify that this surface has
$18$ ordinary cusps and no other singularities.
\end{abstract}

\section{Introduction}
\label{sec:introduction}

The study of singular surfaces in $\mathbb P^3$ is classical.  For surfaces of
degree at most four with only nodes or cusps, that is, singularities of type
$\mathsf A_1$ or $\mathsf A_2$, the extremal cases are known.  In degree four,
the maximum number of nodes is $16$, attained by Kummer quartics, and the
maximum number of cusps is $8$ \cite{Barth_nine_cusps}.  The first degree in
which the corresponding question remains open is degree five.

Let $S\subset \mathbb P^3$ be a complex quintic surface.  The moduli space of
quintic surfaces has dimension $56-1-\dim \operatorname{PGL}(4)=40$.
A node imposes only one condition on this moduli space.  Nevertheless,
Beauville proved that a quintic surface has at most $31$ nodes, and this bound
is sharp \cite{Beauville_nodes_quintics}.

We are interested in quintics with cusp singularities. A cusp is expected to
impose two conditions, thus $20$ cusps impose $40$ conditions on the
moduli space of quintic surfaces. The known upper bound is also $20$
\cite{Tan_cusps}, but it is not known whether this bound is sharp. So far,
the largest known examples are isolated quintic surfaces with $16$ or $17$
cusps \cite{Rito_cuspidal}.

The extremal case of $20$ cusps is especially interesting because of its relation to
canonical maps of surfaces of general type.  Beauville proved that, if the
canonical image is a surface of general type, then the canonical map has
degree at most $9$ \cite{Beauville_canonical}.  But so far only examples
with canonical map of degree at most $5$ are known.  Tan observed that a
quintic with $20$ cusps and the required $3$-divisibility relations could lead
to a smooth $\left(\mathbb Z/3\right)^2$-cover $X$ whose canonical map has
degree $9$ onto the quintic \cite{Tan_cusps}.  Such a surface would satisfy
$p_g(X)=4,\ q(X)=0,\ K_X^2=45,$ hence $K_X^2=9\chi(\mathcal O_X)$.
It would therefore be a ball quotient.  Thus the existence of a quintic with
$20$ cusps is not only an extremal problem about singularities.

In this paper we construct a $6$-dimensional family whose
general member has $17$ cusps and exhibit members with $18$ cusps.

The construction is guided by the cusp code.  Tan's estimate implies that a
quintic with $18$ cusps has a non-zero codeword, whose weight is necessarily
$12$, $15$ or $18$.  The same estimate, together with the Griesmer bound,
shows that a quintic with at least $19$ cusps must contain a word of weight
$12$, see Proposition~\ref{prop:weight-twelve}.  We use the weight-$12$ case,
for which Barth and Rams give an explicit description.

In the general case of the Barth--Rams construction there is a residual plane,
which we take to be
\(
\Pi=\{w=0\}
\),
and the quintic is defined by
\[
F:=\frac{AB-C^3}{w},
\]
where
\[
A:=a^3+wq,\qquad
B:=b^3+wr,\qquad
C:=ab+w\ell.
\]
Here $a,b$ are linear forms on $\Pi$, $q,r$ are quadratic forms on
$\mathbb P^3$, and $\ell$ is a linear form on $\mathbb P^3$. For general
data, the surface ${F=0}$ has a $3$-divisible set of $12$ cusps
\cite{BarthRams_equations}.

One way of exploiting the Barth--Rams form is to require the two contact
cubics to be singular along two skew lines.  After turning two naturally
occurring nodes into cusps, we obtain a family whose general member has $16$
cusps.  Although the singular-line conditions make this family restrictive,
we were able to find quintics with $18$ cusps over $\mathbb F_p$ for small
primes $p$.

Our main construction avoids these singular-line conditions.  Instead, we
require a quintic to admit two Barth--Rams decompositions.  The two associated
sets of $12$ cusps have seven points in common, and their union therefore
consists of $17$ cusps.  After removing the redundancies in the
parametrization, this gives a $6$-dimensional family over $\mathbb Q$ whose
general member has $17$ cusps.  The appearance of one further cusp imposes
only two additional conditions.  Searching in the reduction of this family
over finite fields, we find examples with $18$ cusps.

To obtain an example with $18$ cusps in characteristic zero, we return to the
more special family in which the two contact cubics are singular along skew
lines, which is computationally better adapted to lifting.  Starting from an
$18$-cusp example over $\mathbb F_{17}$, we use Newton--Hensel iteration to
lift the relevant parameters to high $17$-adic precision.  The exact surface
is determined by four parameters. Using LLL we recover polynomial
relations over $\mathbb Q$ among them, and a Gr\"obner basis in lexicographic
order gives their exact algebraic values.  This reconstructs the quintic over
a number field of degree $22$.  We then confirm that the resulting 
surface has $18$ ordinary cusps and no other singularities.

The paper is organized as follows.  Section~\ref{sec:cusp-code} recalls the
cusp code and proves a restriction on the possible codewords of quintics with
many cusps.  Section~\ref{sec:BR} recalls the Barth--Rams decomposition
associated with a $3$-divisible set of $12$ cusps.  Section~\ref{sec:singular-lines}
studies the specialization in which the two contact cubics are singular along
skew lines.  Section~\ref{sec:two-BR} develops the construction using two
Barth--Rams decompositions and obtains the $6$-dimensional family of quintics
with $17$ cusps.  Finally, Section~\ref{sec:lifting} lifts an example with $18$
cusps to characteristic zero and verifies its singularities.

The Magma code for the computations in Sections \ref{sec:two-BR} and \ref{sec:lifting}
is provided in the ancillary files accompanying this paper.

\bigskip
\noindent\textbf{Acknowledgments.}
The second author was financed by Portuguese funds through FCT
(Funda\c c\~ao para a Ci\^encia e a Tecnologia).  This work was carried out
within project UID/00013/2025 of the Centro de Matem\'atica da Universidade
do Minho (CMAT/UM).

\section{The cusp code of a quintic surface}
\label{sec:cusp-code}

Let $S\subset\mathbb P^3$ be a normal quintic surface whose singularities $P_1,\ldots,P_n$ are
cusps, and let $\pi\colon \widetilde S\longrightarrow S$
be the smooth minimal resolution.  Denote by $E_i'$ and $E_i''$ the two exceptional
curves over $P_i$.  The cusp code is the ternary linear code
\[
 \mathcal C(S)=\ker\left(
        \mathbb F_3^n\longrightarrow H^2(\widetilde S,\mathbb F_3),
        \ (c_1,\ldots,c_n)\longmapsto
        \sum_{i=1}^n c_i[2E_i'+E_i'']
        \right).
\]
This is equivalent to the definition using $[E_i'-E_i'']$, since
\[
        [2E_i'+E_i'']=-[E_i'-E_i'']
        \qquad\text{in }H^2(\widetilde S,\mathbb F_3).
\]
The weight of a word is the number of its non-zero coordinates.  The
resolution $\widetilde S$ is simply connected, so the cohomological and
divisor codes considered by Tan coincide in this case
\cite[Section~3.2]{Tan_cusps}.  After interchanging $E_i'$ and
$E_i''$ where necessary, the support $I$ of a non-zero word satisfies
\[
        \sum_{i\in I}\bigl(2E_i'+E_i''\bigr)
        \in 3\operatorname{NS}(\widetilde S).
\]
We say that its support is a {\em $3$-divisible set of cusps}.  It determines
a cyclic triple cover of $S$ branched precisely over the cusps in its support.

Since $K_{\widetilde S}^2=5, \chi(\mathcal O_{\widetilde S})=5, q(\widetilde S)=0$,
Noether's formula gives
\[
        e(\widetilde S)
        =12\chi(\mathcal O_{\widetilde S})-K_{\widetilde S}^2=55,
        \qquad b_2(\widetilde S)=e(\widetilde S)-2=53.
\]
Writing $k=\dim_{\mathbb F_3}\mathcal C(S)$, Tan's estimate
\cite[Lemma~3.2.1]{Tan_cusps} yields
\begin{equation}
 \label{eq:code-dimension}
        k\geq
        \left\lceil\frac{3n-b_2(\widetilde S)}{2}\right\rceil
        =\left\lceil\frac{3n-53}{2}\right\rceil.
\end{equation}
In particular, a quintic with $18$ cusps has a non-zero cusp code.  By
\cite[Theorem~4.3.1]{Tan_cusps}, every non-zero word has weight $12, 15$ or $18$.

\begin{proposition}
 \label{prop:weight-twelve}
Let $S\subset\mathbb P^3$ be a normal quintic surface with only cusp
singularities.  If $S$ has at least $19$ cusps, then its cusp code contains a
word of weight $12$.  Equivalently, its singular set contain a $3$-divisible subset
of $12$ cusps.
\end{proposition}

\begin{proof}
Let $n$ be the number of cusps.  Since $n\leq20$, then
$n=19$ or $n=20$.  From \eqref{eq:code-dimension},
\[
        \dim\mathcal C(S)\geq2\quad\text{if }n=19,
        \qquad
        \dim\mathcal C(S)\geq4\quad\text{if }n=20.
\]
Suppose that $\mathcal C(S)$ contains no word of weight $12$.  Since its only
possible non-zero weights are $12$, $15$ and $18$, its minimum weight $d$
then satisfies $d\geq 15$.

Recall that a ternary $[n,k,d]$ code is an $\mathbb F_3$-linear code
of length $n$ and dimension $k$, with minimum distance $d$.  Since the code is
linear, $d$ is the smallest weight of a non-zero word.  For such a code, the
Griesmer bound gives \cite[Corollary~2.1.4]{Tan_cusps}
\[
        n\geq\sum_{i=0}^{k-1}
        \left\lceil\frac{d}{3^i}\right\rceil.
\]
If $n=19$, then $k\geq2$, so
\[
        19\geq15+\left\lceil\frac{15}{3}\right\rceil=20,
\]
a contradiction.  If $n=20$, then $k\geq4$, so
\[
        20\geq
        15+\left\lceil\frac{15}{3}\right\rceil
        +\left\lceil\frac{15}{3^2}\right\rceil
        +\left\lceil\frac{15}{3^3}\right\rceil
        =23,
\]
again a contradiction.  Therefore $\mathcal C(S)$ contains a word of weight
$12$.
\end{proof}

\section{The Barth--Rams decomposition}
\label{sec:BR}

In this section we recall the Barth--Rams decomposition associated with a $3$-divisible set of $12$ cusps.

Let $S=\{F=0\}\subset\mathbb P^3$ be a normal quintic surface containing a
$3$-divisible set $\mathcal P=\{P_1,\ldots,P_{12}\}$
of cusps, and let $\pi\colon\widetilde S\to S$ be the smooth minimal resolution.  If
$E_i',E_i''$ are the exceptional curves over $P_i$, they may be labelled so
that
\[
 \Lambda_1:=\frac{1}{3}\sum_{i=1}^{12}(2E_i'+E_i''),
 \qquad
 \Lambda_2:=\frac{1}{3}\sum_{i=1}^{12}(E_i'+2E_i'')
\]
are integral divisor classes.  
Writing $H=\pi^*\mathcal O_S(1)$, Barth and Rams show that the classes
$H-\Lambda_1$ and $H-\Lambda_2$ are effective
\cite[Proposition~1.1]{BarthRams_equations}. As a consequence, there are two
contact cubics $\{A=0\}$ and $\{B=0\}$ and a quadric $\{C=0\}$ such that, for suitable
divisors $\Delta_1,\Delta_2$ on $S$,
\[
\operatorname{div}_S(A)=3\Delta_1,
\qquad
\operatorname{div}_S(B)=3\Delta_2,
\qquad
\operatorname{div}_S(C)=\Delta_1+\Delta_2
\]
\cite[Theorem~1.2]{BarthRams_equations}.

It follows that $AB$ and $C^3$ have the same divisor on $S$.  After rescaling
one of the forms, we may therefore assume that $AB=C^3$ on $S$, and hence
\begin{equation}
 \label{eq:BR-decomposition}
        AB-C^3=LF,
\end{equation}
where $L$ is a non-zero linear form
\cite[Lemma~2.1\,(b)]{BarthRams_equations}.  Thus the sextic
$\{AB-C^3=0\}$ is the union of $S$ and the residual plane
$\Pi=\{L=0\}$.
We call \eqref{eq:BR-decomposition} a Barth--Rams decomposition of $S$.

We use the dense open case of their description.  Choose coordinates $(x,y,z,w)$ so that
$L=w$.  On the residual plane $\Pi=\{w=0\}$ one has the restriction
$\overline A\,\overline B=\overline C^{\,3}.$
For a general member, $\overline C$ is the product of two distinct linear
forms $a,b$ on $\Pi$, and, after rescaling and possibly interchanging $A$ and
$B$,
\[
        \overline A=a^3,
        \qquad
        \overline B=b^3,
        \qquad
        \overline C=ab.
\]
The data can therefore be written as
\begin{equation}
 \label{eq:BR-normal-form}
        A=a^3+wq,
        \qquad
        B=b^3+wr,
        \qquad
        C=ab+w\ell,
\end{equation}
where $q,r$ are quadratic forms and $\ell$ is a linear form.  The corresponding
quintic equation is
\begin{equation}
 \label{eq:BR-quintic}
\begin{aligned}
        F&=\frac{AB-C^3}{w}.
\end{aligned}
\end{equation}

Barth and Rams also describe boundary cases. All computations below take place
in the open family \eqref{eq:BR-normal-form}.

The cusps arising from the decomposition are the points of the variety $V(A,B,C)$
outside $\Pi$.  At a transverse point $P$ of this intersection, the functions $A,B,C$ are local
coordinates and $w(P)\ne0$, so the local equation of $S$ is analytically
\[
        uv-t^3=0.
\]
Hence $P$ is an ordinary cusp.  The complete intersection $V(A,B,C)$ has
degree $18$. For general data, intersection multiplicity $6$ is concentrated
on the residual plane and the surface $S$ has $12$ ordinary $3$-divisible cusps
\cite[Theorem~2.1]{BarthRams_equations}.

Combining Proposition~\ref{prop:weight-twelve} with the Barth--Rams theorem,
every quintic with $19$ or $20$ cusps admits a decomposition of the form
\eqref{eq:BR-decomposition}.

\section{A specialization by singular contact cubics}
\label{sec:singular-lines}

Before introducing the main construction, we describe a specialization
of the Barth--Rams family.

The presentation \eqref{eq:BR-normal-form} depends
on $30$ coefficients: three for each of $a$ and $b$, ten for each of $q$ and
$r$, and four for $\ell$.  These are not moduli.  After quotienting by the
$12$-dimensional subgroup of $\operatorname{PGL}(4)$ preserving the residual
plane and by the two natural rescalings of the decomposition, the family has
dimension $16$.
Directly imposing six further moving cusps in this presentation leads to a
large system.  We therefore first look for a geometric specialization in which
some additional singularities occur automatically.

Keep the residual plane $\Pi=\{w=0\}$
and require the contact cubics $\{A=0\}$ and $\{B=0\}$ to be singular along the skew
lines
\[
        \Gamma_1=\{x=y=0\},
        \qquad
        \Gamma_2=\{x-z=y-w=0\},
\]
respectively.  The twelve Barth--Rams cusps remain.  In addition, the quadric
$\{C=0\}$ meets each line $\Gamma_i$ in two points.  One of them lies outside
$\Pi$.  Since the corresponding cubic is singular along $\Gamma_i$ and $\{C=0\}$
cuts the line transversely, this point is an ordinary cusp of the quintic.
Thus the two lines give two further cusps.

The other intersection points lie on the residual plane.  These two points are
\[
        Q_1=\Gamma_1\cap\Pi=(0:0:1:0),
        \qquad
        Q_2=\Gamma_2\cap\Pi=(1:0:1:0).
\]
Locally, the sextic $\{AB-C^3=0\}$ contains the component $\{w=0\}$, 
and the quintic equation is obtained by dividing by $w$.  
For a general member of the specialized family, the
resulting quadratic part at each $Q_i$ is non-degenerate.  Hence $Q_1$ and
$Q_2$ are nodes, and the general surface has singular set of type
$14\mathsf A_2+2\mathsf A_1.$

After the singular-line conditions are imposed, the construction depends on
$14$ parameters, some of which are redundant.  We impose the vanishing of the
Hessian determinant at $Q_1$ and $Q_2$, thus giving
a $12$-parameter (not moduli) family of quintic surfaces whose general member has $16$ cusps.

Since each additional cusp imposes two conditions, the locus with $18$ cusps
has expected codimension $4$ inside this family.  In fact, over small finite
fields we found quintics with $18$ cusps.  This specialization is nevertheless
restrictive: requiring both contact cubics to be singular along prescribed
lines cuts out a part of the Barth--Rams family.  Our aim is
to construct a larger family with many cusps.  We therefore replace the
singular-line conditions by the more flexible requirement that the same
quintic admits two Barth--Rams decompositions.

\section{Two Barth--Rams decompositions}
\label{sec:two-BR}

We now impose two Barth--Rams decompositions on the same quintic.  We first
use projective equivalence to take the two residual planes to be
$\Pi_w=\{w=0\}$ and $\Pi_z=\{z=0\}$.  Let $a,b$ be the two linear forms
on $\Pi_w$ associated with the first Barth--Rams decomposition.  We work on
the dense open subset where $w,z$, together with two linear extensions of
$a,b$ to $\mathbb P^3$, are linearly independent.  By projective equivalence,
we may then take these extensions to be $x$ and $y$.  Thus $a=x$ and $b=y$
on $\Pi_w$, and the first decomposition is
\[
\begin{aligned}
        A_w&=x^3+wq_1,\\
        B_w&=y^3+wq_2,\\
        C_w&=xy+w\ell,
\end{aligned}
\qquad
        F_w=\frac{A_wB_w-C_w^3}{w},
\]
where $q_1,q_2$ are quadratic forms and $\ell$ is a linear form.

For the second decomposition, write
\[
        \alpha=A_1x+A_2y+A_3w,
        \qquad
        \beta=B_1x+B_2y+B_3w,
\]
and let $Q_1,Q_2$ be quadratic forms and $L$ a linear form.  We set
\[
\begin{aligned}
        A_z&=\alpha^3+zQ_1,\\
        B_z&=\beta^3+zQ_2,\\
        C_z&=\alpha\beta+zL,
\end{aligned}
\qquad
        F_z=\frac{A_zB_z-C_z^3}{z}.
\]
The condition that the two decompositions define the same quintic is the
polynomial identity
\begin{equation}
 \label{eq:two-BR-identity}
        F_w=F_z.
\end{equation}
Comparing coefficients gives the equations used in the Magma computation.

The first normalized decomposition depends on $24$ coefficients and the
second on $30$, so before imposing \eqref{eq:two-BR-identity} we have $54$
parameters.  There is still a $5$-dimensional group of projective
transformations preserving the two residual planes and the normalization of
the first decomposition.  Its elements have the form
\begin{equation}
 \label{eq:residual-projectivities}
 (x,y,z,w)\longmapsto
 (\lambda x+uw,\,\mu y+vw,\,\nu z,\,\rho w),
\end{equation}
modulo a common scalar.  There is also the one-dimensional rescaling
\begin{equation}
 \label{eq:second-BR-rescaling}
 \begin{aligned}
        \alpha&\longmapsto t\alpha,
        &Q_1&\longmapsto t^3Q_1,\\
        \beta&\longmapsto t^{-1}\beta,
        &Q_2&\longmapsto t^{-3}Q_2,
        &L&\longmapsto L,
 \end{aligned}
\end{equation}
which leaves $F_z$ unchanged.

We use these six degrees of freedom to normalize the second decomposition.
Let $L_3$ denote the coefficient of $z$ in $L$.
On the dense open subset where
\[
 A_1A_2B_1L_3(A_1B_2-A_2B_1)\ne0,
\]
the translations in
\eqref{eq:residual-projectivities} first give
\(
        A_3=B_3=0
\).
The diagonal coordinate changes, together with
\eqref{eq:second-BR-rescaling}, can then be used to impose
\(
        A_1=A_2=B_1=1
\),
and the remaining scaling of $z$ gives $L_3=1$.  Thus we may write
\begin{equation}
 \label{eq:normalized-second-BR}
        \alpha=x+y,
        \qquad
        \beta=x+B_2y,
        \qquad
        L=L_1x+L_2y+z+L_4w,
\end{equation}
with $B_2\ne1$.  Writing
\[
 Q_1=\sum_{i=1}^{10}C_i m_i,
 \qquad
 Q_2=\sum_{i=1}^{10}D_i m_i,
\]
where
\[
 (m_1,\ldots,m_{10})=
 (x^2,xy,xz,xw,y^2,yz,yw,z^2,zw,w^2),
\]
the normalized system therefore depends on $48$ parameters: 
the $24$ coefficients of $q_1,q_2,\ell$, the $20$ coefficients
of $Q_1,Q_2$, the coefficient $B_2$, and $L_1,L_2,L_4$.  Comparing
coefficients in \eqref{eq:two-BR-identity} gives, in our implementation,
$56$ polynomial equations in these $48$ parameters.

Over finite fields, the family obtained above is very easy to use
computationally.  A direct computation of a general member was initially
rather slow, but a closer study of the defining equations showed that the
choice of variables to solve for is crucial.  With a suitable choice, the same
computation becomes very fast.  For random choices over finite fields, we
verified that the resulting member is a quintic surface with exactly $17$
cusps, and that the two sets of $12$ cusps determined by its Barth--Rams
decompositions intersect in $7$ points.

We fixed one such quintic $Q$ over the rational numbers.  
At the corresponding point $p$ of the normalized
coefficient scheme, the Jacobian matrix of the $56$ equations has rank
$42$.
Consequently the Zariski tangent space at $p$ has dimension
$48-42=6$.
Our goal is to show that the irreducible component through $p$ has dimension
exactly $6$.

\begin{lemma}
\label{lem:three-divisibility-deformation}
Let $X$ be a quintic surface with cusps, and let $\mathcal N$ be a
$3$-divisible set of cusps.  Then $\mathcal N$ remains $3$-divisible under
every sufficiently small equisingular deformation of $X$.
\end{lemma}

\begin{proof}
Let $\widetilde X\to X$ be the minimal resolution.  After choosing the
labelling of the two exceptional curves over each cusp, let $D$ be the
exceptional divisor corresponding to $\mathcal N$.  By $3$-divisibility there
is a divisor $M$ on $\widetilde X$ such that
\[
        D\equiv 3M,
\]
where $\equiv$ denotes linear equivalence.
Since an $A_2$-singularity is rational, an equisingular deformation admits a
simultaneous resolution, see \cite[Theorem~5.16]{Wahl76}.  Thus, after
shrinking the base, the exceptional curves deform with the surface and give
divisors $D_t$ on the nearby resolutions $\widetilde X_t$.

The simultaneous resolution is a smooth proper family, so over a sufficiently
small disk its fibres are differentiably identified.  We may therefore
identify
\[
        H^2(\widetilde X_t,\mathbb Z)
        \simeq H^2(\widetilde X,\mathbb Z),
\]
with the classes of the exceptional curves corresponding under this
identification.  Hence
\[
        [D_t]=3\alpha_t
\]
for an integral class $\alpha_t\in H^2(\widetilde X_t,\mathbb Z)$.  Since
$D_t$ is a divisor, $[D_t]$ is of type $(1,1)$, and therefore so is
$\alpha_t=[D_t]/3$.  By the Lefschetz $(1,1)$ theorem there is a line bundle
$M_t$ on $\widetilde X_t$ with $c_1(M_t)=\alpha_t$.  Finally,
$\widetilde X_t$ is regular, hence $\operatorname{Pic}^0(\widetilde X_t)=0$,
so equality of first Chern classes gives
\[
        D_t\equiv 3M_t.
\]
Thus the same set of moving cusps remains $3$-divisible.
\end{proof}

\begin{theorem}
\label{prop:two-BR-dimension}
Let $Q$ be the quintic with $17$ cusps fixed above.
The irreducible component of the two Barth--Rams locus containing $Q$ has 
dimension $6$ in the moduli space of quintic surfaces. In particular, this
construction gives a $6$-dimensional family of quintic surfaces with $17$
cusps.
\end{theorem}

\begin{proof}
Let $\mathcal U$ be the locus of quintic surfaces admitting at least two
Barth--Rams decompositions, and let $\mathcal M_{17}$ denote the equisingular
locus of all quintics with $17$ cusps near $Q$.  The two decompositions of $Q$
give two $3$-divisible sets of $12$ cusps.  By
Lemma~\ref{lem:three-divisibility-deformation}, both sets remain
$3$-divisible under every sufficiently small deformation of $Q$ inside
$\mathcal M_{17}$.  By the Barth--Rams construction of
Section~\ref{sec:BR}, they therefore continue to give two Barth--Rams
decompositions.  Hence there is a neighbourhood $V$ of $Q$ in the
moduli space of quintic surfaces such that
\[
        \mathcal M_{17}\cap V\subset \mathcal U\cap V.
\]

The locus $\mathcal M_{17}$ has local dimension at least $6$.  Indeed, a
moving $A_2$-point imposes at most two conditions on the $40$-dimensional
moduli space of quintic surfaces: locally one imposes the four first-derivative
equations and one Hessian equation, while the moving point itself has three
parameters.  Thus $17$ cusps impose at most $34$ conditions. We have then
\[
        \dim_Q\mathcal U\ge \dim_Q\mathcal M_{17}\ge 40-34=6.
\]

On the other hand, after possibly shrinking $V$, every member of
$\mathcal U\cap V$ can be normalized as above.  Thus $\mathcal U\cap V$ is
contained in the image of a neighbourhood of $p$ in the $48$-variable
coefficient scheme.  The Jacobian of this scheme has rank $42$ at $p$, so
\[
        \dim_Q\mathcal U
        \le \dim_p Z
        \le \dim T_pZ
        =48-42=6,
\]
where $Z$ denotes the normalized coefficient scheme.
Therefore
\[
        \dim_Q\mathcal U=6.
\]
In particular, the component of $\mathcal U$ through $Q$ coincides locally
with a component of the $17$-cusp equisingular locus.
\end{proof}

Since one further cusp imposes only two conditions, it is fast to find
quintic surfaces with $18$ cusps over finite fields, by a random search.
Lifting one of these examples to characteristic
zero, however, turned out to be considerably more difficult.

For this purpose the more special family of
Section~\ref{sec:singular-lines} is computationally better adapted.  In the
next section we start from an $18$-cusp member of that family over
$\mathbb F_{23}$ and explain how it can be lifted and reconstructed exactly
over a field of characteristic zero.

\section{Lifting an $18$-cusp quintic to characteristic zero}
\label{sec:lifting}

Our goal is now to construct an explicit quintic surface with $18$ cusps over
a field of characteristic zero.  We start from an example over
$\mathbb F_{23}$ in the specialized family of
Section~\ref{sec:singular-lines}.  Such an example is easily found by a random
search over $\mathbb F_{23}$.

\subsection{Newton--Hensel lifting}

We first turn the finite-field example into a square system suitable for
$p$-adic Newton iteration.  Recall that in the specialization of
Section~\ref{sec:singular-lines} the two contact cubics are singular along two
prescribed skew lines $\Gamma_1$ and $\Gamma_2$.  The construction already
contains the twelve Barth--Rams cusps and two further cusps, one on each
singular line and outside the residual plane $\Pi$.  The other intersection
points of the singular lines with the quadric lie on the residual plane and
are
\[
        Q_1=\Gamma_1\cap\Pi,\qquad
        Q_2=\Gamma_2\cap\Pi.
\]
For a general member of the family these two points are nodes.  We impose the
vanishing of the Hessian determinant at $Q_1$ and $Q_2$ in order to turn them
into cusps, and we impose the existence of two extra cusps.  After fixing four
surface coefficients, the resulting system has twelve unknowns: six remaining
surface coefficients and the three affine coordinates of each of the two
extra cusps.  There are also
twelve equations: two Hessian equations for the distinguished nodes, and, for
each extra cusp, the four first-derivative equations together with the
vanishing of the determinant of the affine Hessian.

Let
\[
        \boldsymbol f=(f_1,\ldots,f_{12})
\]
be this system and let
\[
        \boldsymbol u_0\in\mathbb F_p^{12},
        \qquad p=23,
\]
be the solution corresponding to our finite-field surface.  In our example
the Jacobian matrix
\[
        J(\boldsymbol u_0)
        =
        \left(
        \frac{\partial f_i}{\partial u_j}(\boldsymbol u_0)
        \right)
\]
is invertible modulo $p$.

We lift the solution by Newton iteration, doubling the $p$-adic precision at
each step.  Suppose that $\boldsymbol u_n$ is known modulo
$p^{2^n}$.  Put $e=2^n$ and look for a correction of the form
\[
        \boldsymbol u_{n+1}
        =
        \boldsymbol u_n+p^e\boldsymbol\delta_n.
\]
Taylor expansion gives
\[
\boldsymbol f(\boldsymbol u_n+p^e\boldsymbol\delta_n)
\equiv
\boldsymbol f(\boldsymbol u_n)
+p^eJ(\boldsymbol u_n)\boldsymbol\delta_n
\pmod{p^{2e}}.
\]
Therefore it is enough to solve
\begin{equation}
 \label{eq:newton-lifting}
 J(\boldsymbol u_n)\boldsymbol\delta_n
 \equiv
 -\frac{\boldsymbol f(\boldsymbol u_n)}{p^e}
 \pmod {p^e}.
\end{equation}
Since the Jacobian is invertible modulo $p$, this linear system has a unique
solution, and the corrected vector satisfies
\[
        \boldsymbol f(\boldsymbol u_{n+1})
        \equiv0\pmod {p^{2e}}.
\]
Thus the successive precisions are
\[
        p,\ p^2,\ p^4,\ p^8,\ldots,\ p^{2^n}.
\]
If the desired final precision is $p^N$ with $N$ not a power of two, the same
procedure is used, replacing the last doubling step by the required final
exponent.

\subsection{Reduction to four parameters and LLL}

The full Newton system has twelve variables, but only six of them are surface
coefficients. The other six are the coordinates of the two extra cusps and
are auxiliary for the reconstruction of the surface.  Moreover, the two
Hessian equations which turn the two distinguished nodes into cusps determine
two of these six surface coefficients in terms of the others.  Consequently,
the exact surface is determined by only four $23$-adic numbers, which we
denote by $s_1,s_2,s_3,s_4$.

We lift these numbers to precision $23^{2000}$ and use LLL to
recognize polynomial relations over $\mathbb Q$ among them.  Here LLL stands
for the Lenstra--Lenstra--Lov\'asz lattice-basis reduction algorithm \cite{LLL82}.  It
replaces a given basis of an integer lattice by a reduced basis containing
comparatively short vectors.

Clearing denominators, it is enough to search for relations with integral
coefficients.  Let
\[
        m_1=1,m_2,\ldots,m_{210}
\]
be the monomials in variables $(x,y,z,w)$ of total degree at most $6$. Their
number is 210.
For the LLL search we use the precision
\(
        M_s=23^{1900}
\)
and choose integers $v_i$ representing
\[
        v_i\equiv m_i(s_1,s_2,s_3,s_4)\pmod {M_s}.
\]
We seek small integers $c_1,\ldots,c_{210}$ such that
\[
        \sum_{i=1}^{210}c_iv_i\equiv0\pmod {M_s}.
\]
They give the candidate polynomial relation
\(
        \sum_{i=1}^{210}c_im_i=0
\).

Since $v_1=1$, the lattice of all coefficient vectors satisfying this
congruence is the row lattice of the matrix
\begin{equation}
\label{eq:lll-lattice}
 B=
 \begin{pmatrix}
 M_s      &0&0&\cdots&0\\
 -v_2     &1&0&\cdots&0\\
 -v_3     &0&1&\cdots&0\\
 \vdots   &\vdots&\vdots&\ddots&\vdots\\
 -v_{210} &0&0&\cdots&1
 \end{pmatrix}.
\end{equation}
In the implementation, the entries $-v_i$ in the first column are replaced
by their residues modulo $M_s$, which gives the same row lattice.
Every vector $\boldsymbol c=(c_1,\ldots,c_{210})$ in the row lattice of $B$
satisfies
\[
        c_1+\sum_{i=2}^{210}c_iv_i\equiv0\pmod {M_s},
\]
and conversely every integral vector satisfying this congruence belongs to
the row lattice.  LLL reduction of~\eqref{eq:lll-lattice} therefore produces
short coefficient vectors and hence candidate relations.  Exact relations
with moderate coefficients remain short as the precision increases, whereas
accidental modular relations normally require much larger coefficients.

The candidates found using $M_s=23^{1900}$ are then checked independently at
the full available precision $23^{2000}$.  Thus the final one hundred
$23$-adic digits are withheld from the LLL computation and used only for
validation.  In the successful computation, the validated relations of total
degree at most $6$ give enough equations to isolate the required algebraic
solution.

Let $I\subset\mathbb Q[x,y,z,w]$
be the ideal generated by the validated relations obtained from the degree-$6$
search.  We then compute a Gr\"obner basis of $I$ with respect to lexicographic order.  The
resulting zero-dimensional triangular system gives the exact algebraic values
of the four parameters, which are defined over a number field of degree 22.
Substituting these values in the formulas of
Section~\ref{sec:singular-lines} reconstructs the contact cubics, the quadric,
and hence the quintic exactly over a number field.  Thus the finite field
surface has been lifted to an explicit surface in characteristic zero.

\subsection{Verification of the singularities}
\label{sec:verification}

\begin{theorem}
The quintic surface reconstructed above has exactly $18$ ordinary cusps and no
other singularities.
\end{theorem}

\begin{proof}
Let $K$ be the number field obtained in the reconstruction and let
\[
        S=\{F=0\}\subset\mathbb P^3_K
\]
be the resulting quintic.  We first compute, over $K$, its
singular subscheme $\Sigma$.
The computation gives
\[
        \dim\Sigma=0,\qquad \deg\Sigma=36.
\]

We next reduce the exact surface modulo $p=1031$, obtaining a surface over
the degree-$22$ extension $\mathbb F_{1031^{22}}$.
Its singular subscheme is zero-dimensional.
Its primary components have degrees
\[
        8,6,6,2,2,2,2,2,2,2,2,
\]
while the corresponding reduced components have degrees
\[
        4,3,3,1,1,1,1,1,1,1,1.
\]
In particular, the reduced singular locus consists of exactly $18$ singular points.

The last eight points in the above decomposition are directly visible from
the construction and are cusps.  We now verify the twelve cusps coming from
the Barth--Rams decomposition.  Recall that, in the decomposition
\[
        ab-c^3=zF,
\]
the equations $a=0$ and $b=0$ define the two contact cubic surfaces, while
$c=0$ is the quadric.  The Barth--Rams points are the points of the complete
intersection $V(a,b,c)$ outside the residual plane $z=0$.

What remains to be checked is that this intersection is transverse at those
points, since this is precisely the condition which gives ordinary cusps in
the Barth--Rams construction.  Intersecting the reduced primary components
of the singular locus with $V(a,b,c)$ gives degrees
\[
        4,3,3,1,1,0,0,0,0,0,0.
\]
Hence the intersection consists of $12$
reduced points.  Thus the quadric $c=0$ meets the curve $a=b=0$
transversally at these twelve points, and the local Barth--Rams equation is
of type $\mathsf A_2$.  Therefore the quintic contains $18$ ordinary cusps and
no other singularities.

Finally, an $\mathsf A_2$ singularity contributes
length $2$ to the singular subscheme, so these eighteen cusps already
contribute $36$.
Since the singular subscheme over $K$ has degree $36$,
the characteristic-zero surface contains the corresponding
eighteen ordinary cusps and no other singularities.

\end{proof}

\bibliographystyle{amsplain}
\bibliography{References}

\providecommand{\bysame}{\leavevmode\hbox to3em{\hrulefill}\thinspace}
\providecommand{\MR}{\relax\ifhmode\unskip\space\fi MR }
\providecommand{\MRhref}[2]{%
  \href{http://www.ams.org/mathscinet-getitem?mr=#1}{#2}
}
\providecommand{\href}[2]{#2}
\begin{thebibliography}{1}

\bibitem{Barth_nine_cusps}
Wolf~P. Barth, \emph{On the classification of {$K3$} surfaces with nine cusps},
  Complex Analysis and Algebraic Geometry (Thomas Peternell and Frank-Olaf
  Schreyer, eds.), Walter de Gruyter, Berlin, 2000, pp.~41--59.

\bibitem{BarthRams_equations}
Wolf~P. Barth and S{\l}awomir Rams, \emph{Equations of low-degree projective
  surfaces with three-divisible sets of cusps}, Math. Z. \textbf{249} (2005),
  no.~2, 283--295.

\bibitem{Beauville_canonical}
Arnaud Beauville, \emph{L'application canonique pour les surfaces de type
  g\'en\'eral}, Invent. Math. \textbf{55} (1979), no.~2, 121--140.

\bibitem{Beauville_nodes_quintics}
\bysame, \emph{Sur le nombre maximum de points doubles d'une surface dans
  {$\mathbb P^3$} ({$\mu(5)=31$})}, Algebraic Geometry, Angers, 1979, Sijthoff
  \& Noordhoff, Alphen aan den Rijn, 1980, pp.~207--215.

\bibitem{LLL82}
A.~K. Lenstra, H.~W.~jun. Lenstra, and L{\'a}szl{\'o} Lov{\'a}sz,
  \emph{Factoring polynomials with rational coefficients}, Math. Ann.
  \textbf{261} (1982), 515--534.

\bibitem{Rito_cuspidal}
Carlos Rito, \emph{Cuspidal quintics and surfaces with {$p_g=0$}, {$K^2=3$} and
  {$5$}-torsion}, LMS J. Comput. Math. \textbf{19} (2016), no.~1, 42--53.

\bibitem{Tan_cusps}
Sheng-Li Tan, \emph{Cusps on some algebraic surfaces and plane curves}, Complex
  Analysis, Complex Geometry and Related Topics---Namba 60, 2003, pp.~106--121.

\bibitem{Wahl76}
Jonathan~M. Wahl, \emph{Equisingular deformations of normal surface
  singularities. {I}}, Ann. Math. (2) \textbf{104} (1976), 325--356.

\end{thebibliography}

\vspace{1cm}

\noindent Lev Borisov \vspace{0.1cm} 
\\Department of Mathematics, Rutgers University
\\Piscataway, NJ 08854
\\ \verb|borisov@math.rutgers.edu| 

\vspace{1cm}

\noindent Carlos Rito
\vspace{0.1cm}
\\ Centro de Matem\'atica, Universidade do Minho -- Polo CMAT-UTAD
\vspace{0.1cm}
\\ Universidade de Tr\'as-os-Montes e Alto Douro, UTAD
\\ Quinta de Prados
\\ 5000-801 Vila Real, Portugal
\vspace{0.1cm}
\\ \url{https://www.utad.pt}, \texttt{crito@utad.pt}

\end{document}